\documentclass{aims}%
\usepackage{amsmath}
  \usepackage{paralist}
 \usepackage{enumitem}
  \usepackage{graphics} 
  \usepackage{epsfig} 
\usepackage{graphicx}  \usepackage{epstopdf}
 \usepackage[colorlinks=true]{hyperref}
 \usepackage{tikz}
 \usepackage{color} 
 \usepackage{enumitem}
\hypersetup{urlcolor=blue, citecolor=red}
\usepackage{lineno}
  
\def\RR{\mathbb{R}}
\newtheorem{theorem}{Theorem}[section]

\newtheorem{lemma}[theorem]{Lemma}
\newtheorem{proposition}[theorem]{Proposition}

\theoremstyle{definition}
\newtheorem{definition}[theorem]{Definition}
\newtheorem{remark}{Remark}
\newtheorem*{notation}{Notation}
\newtheorem{Assumptions}{Assumptions}
\newtheorem{Assumption}[theorem]{Assumption}
\newcommand{\mc}{\mathcal}

\title[Origin-to-destination network flow with path preferences] 
      {Origin-to-destination network flow with path preferences and velocity controls: a mean field game-like approach}

\author[Fabio Bagagiolo, Rosario Maggistro, Raffaele Pesenti]{}

\subjclass{Primary: 91A13, 49J15, 90B10, 90B20;
	 Secondary: 91A25, 91B08.
 }
 \keywords{Traffic flow optimal control, mean field games, path preference dynamics, existence of equilibrium, dynamical flow networks.}

 \email{fabio.bagagiolo@unitn.it}
 \email{rosario.maggistro@unive.it}
 \email{pesenti@unive.it}

\thanks{$^*$ Corresponding author: Rosario Maggistro}

\begin{document}
\maketitle

\centerline{\scshape Fabio Bagagiolo}
\medskip
{\footnotesize
 \centerline{Department of Mathematics}
   \centerline{Universit\`{a} di Trento}
   \centerline{ Via Sommarive, 14, I-38123 Povo, Trento, Italy}
} 

\medskip

\centerline{\scshape Rosario Maggistro$^*$ and Raffaele Pesenti}
\medskip
{\footnotesize
 \centerline{ Department of Management}
   \centerline{ Universit\`{a} Ca' Foscari Venezia}
   \centerline{	Fondamenta S. Giobbe, 873, I-30121 Cannaregio, Venezia, Italy}
}

%

\bigskip


\begin{abstract}
In this paper we consider a mean field approach to modeling the agents flow over a transportation network. In particular, beside {a standard framework of mean field games, with controlled dynamics by the agents and costs mass-distribution dependent}, we also consider a path preferences dynamics obtained as a generalization of the {so-called} noisy best response dynamics. Such a preferences dynamics says the agents choose their path having access to global information about the network congestion state and based on the observation of the decision of the agents that have preceded. {We prove the existence of a mean field equilibrium obtained as a fixed point of a map over a suitable set of time-varying mass-distributions, defined edge by edge in the network. 
We also address the case where the admissible set of controls is suitably bounded depending on the mass-distribution on the edge itself.}

	
\end{abstract}

\section{Introduction}

In this paper,  we introduce a Mean Field approach to modeling and analytically studying the agents flow over a transportation network. 

We frame our work in the literature on the flow dynamics of agents, which have become in the last decades of interest for several research communities. In the transportation area, for example, the interest towards such topics is due to the continuous growth of traffic flow as well as the spread of information and communication technologies which are changing the transportation system dynamics and affecting the users' decision making and behaviors.
Different modeling approaches have been proposed which can generally be
classified into three categories: microscopic, macroscopic and multi-scale models. 
The microscopic models or ``individual based
models", describe the crowd by giving the dynamics of each agent, usually via an ordinary differential equation and are particularly well suited for use
with small crowds. Such approach includes the cellular automaton model (see e.g., \cite{Burstedde}), the lattice gas model (see e.g. \cite{Guo}) and the social force model considered in \cite{Helbing2001}. Specifically, in \cite{Helbing2001}, the authors introduce the concept of social force to measure the internal motivation of
the individuals in performing certain movements. 
Another microscopic description is
provided in \cite{Hoogendoorn 2004}--\cite{Hoogendoorn2005}, where a theory of pedestrian route choice behavior based on the concepts of walking task and walking cost is proposed. Each pedestrian plans her movements
on the basis of some predictions she makes on the other
individuals' behavior. She makes her decisions by minimizing her individually estimated
walking cost, expressed by a functional depending on the predicted
positions of other people.

Macroscopic models, in contrast, focus on the overall behavior of pedestrian flows and are more suited to investigations of extremely large crowds, especially
when examining aspects of motion in which individual differences are less important. Such models describe the evolution of the population's density through a partial differential equation, often of transport type. In \cite{Hughes} the crowds is treated as a ``thinking fluid"  and the model is described by the continuity equation coupled with the eikonal equation. In \cite{Bellomo2008}, instead, the continuity equation is linked to the linear momentum one. Both models are based one the concepts of preferred direction of motion and discomfort at high densities.
In the framework of scalar conservation laws, a macroscopic
one-dimensional model has been proposed in \cite{Colombo2005} with the aim of describing the transition from normal to panic conditions. Finally, in \cite{PiccoloTosin2011} a new model of pedestrian flow, formulated within a measure-theoretic framework is proposed. It consists of a representation of the system via a family of measures which provide an estimate of the space occupancy by pedestrians at successive times. 

The multi-scale models use measure evolution equations for describing crowds
mixing a microscopic and a macroscopic description. In particular, the time evolving measure allow to split the density into a microscopic granular and a macroscopic continuous mass. 
These kinds of multi-scale models were  introduced quite recently for crowd and pedestrian dynamics modeling (see \cite{Crisitani2011}--\cite{CrisitaniPiccoli2014}, \cite{Rossi2013}--\cite{Rossi2018}) and enjoy the following properties.
They are able to capture some typical phenomena such as self organization.
Their different scales can be used to model the relative importance of agents in a crowd: 
for example, in a leader-follower system, leaders are described by a precise microscopic model, 
while followers are taken into account by the macroscopic part.

{In \cite{bagpes}, \cite{bafama} a mean field game approach is implemented for studying the optimal behavior of agents flowing on a network having more than one target (vertices of the networks) to be reached (visited). In \cite{CDC} an origin-destination model with path preferences dynamics as the one here presented is preliminary treated.
In the present paper, starting from the similar analysis of the different problem in \cite{bafama}, and generalizing the results in \cite{CDC},} beside the usual framing of mean field games 
(typically defined by the pair made of Hamilton-Jacobi-Bellman and mass conservation
equations), we also consider the agent's path preferences dynamics.
Specifically, we propose a model in which
the agents choose their path having access to global
information about the network congestion, but also being influenced by the decision of agents that has already made their decisions.	

Then, our model consider two dynamics:
the first one based on the mass conservation equations describes the real time evolution of the congestion level in each edge of the network;
the second one involves the evolution of the agents' path preferences. 
It is related to the agents' experiences and the available information. It evolves at a slow time scale as compared to the first one.\\
One possible physical interpretation of our model is to consider the agents as pedestrians traversing possible paths within a city described as a network. 
{However, it may also seen as well suited to describe, for example, car traffic flow in highways networks.}
In this way, the model can be 
related to two streams of literature on transportation networks.
On the one hand, {pedestrians} flows on networks have been widely analysed using the different modeling approaches cited above. As compared
to the macroscopic and multi-scale approaches (typically described by partial differential equations),
ours significantly simplifies the evolution of the traffic masses (using a balance ordinary differential equations), whereas it
highlights the role of agents route choice behavior which is typically neglected in that literature.
On the other hand, transportation networks have been studied from a decision theoretic
perspective within the framework of congestion games \cite{Beckmann}, \cite{Rosenthal}. In this framework, however, the information is available to the agents at
a single temporal and spatial scale and the mass conservation equations are
completely neglected by assuming that they are instantaneously equilibrated.
In contrast, we study a model where the
mass conservation equations are not neglected and agents route choice decisions
are affected both by the global information on the congestion and by 
the decision of the agents that have preceded entering the network.	

As already mentioned our models is based on Mean field games (MFG), whose theory goes back to the seminal 
work by Lasry-Lions \cite{LLions} (see also \cite{HCAINMAL}). This theory
includes methods and techniques to study differential games with a large
population of rational players and it is based on the assumption that the
population influences individuals' strategies through mean field
parameters. Several application domains such as economics, physics, biology
and network engineering accommodate MFG theoretical models (see
\cite{AcCamDolc}, \cite{Guent2011}, \cite{LachSalTur}--\cite{LachWolf}). In particular, models to 
study of dynamics on networks and/or pedestrian movement can be found for
example in \cite{CCMar}, \cite{CPTos}, \cite{CDeMTos}, \cite{BBMZop}.

\noindent{Beside the position of the problem, which is also rather new, the main goal of the present paper is to prove the existence of a mean field equilibrium for our framework.} 
This equilibrium is a time-varying distribution of agents ${\rho}$, 
{defined edge by edge in the network}, that generates an optimal controls vector  which, in turn, yields a path preference vector providing once again the time-varying distribution~${\rho}$. It is obtained as a fixed point of a map which satisfies the conditions of the Brouwer fixed-point theorem.  
{In out model, the controls implemented by an agent can be interpreted as the the velocities at which the agent traverses the network edges. 
Then, we also address the case where a mass-distribution dependent bound on the set of admissible controls is assumed, in order to take account of possible constraints in the velocities when edges are very congested.} 

The rest of this paper is organized as follows. 
In Section \ref{sec:2}, we describe the model and state the hypotheses used in the paper. Moreover we separately analyse all the agents' dynamics  which constitute our transportation system.
In Section \ref{sec:3}, we prove the existence of a mean-field equilibrium
and, {in Section \ref{sec:4}, we study a new mean field game problem with a constraint on the set of admissible controls.}\\
In Section \ref{sec:5}, we draw conclusions and suggests future works.
\begin{notation}
Hereinafter, in the paper we will use the following notation.

\begin{tabular}{ll}
${\mc V}$ & the finite set of vertices; \\
${\mc E}$ & the finite set of directed edges; \\
$e$ & the index of the edge;\\
$p$ &  the index of path;\\
$o$ & the origin vertex;\\
$d$ & the destination vertex;\\
	$\nu_e$ & the tail vertex of the edge $e$; \\
	$\kappa_e$ & the head vertex of the edge $e$;\\
	$\ell_e$ & the length of edge $e$; \\
	$u^e_p(\cdot)$ & {the measurable control for agents in the edge $e\in p$}; \\
	$u^e_p[t]$ & {the optimal constant control chosen at starting time $t$ for traversing $e\in p$};  \\
	$C_e$ & the {maximal mass} of agents that can enter in $e$ per unit of time;\\
	$\rho_{\max}$ & the {maximal mass} of agents that can be present at the same time in $e$;\\
	$\Gamma$ & the set of all the paths $p$ from $o$ to $d$;\\
	$A$ & the edge-path incidence matrix (see \eqref{incidencematrix}); \\
	$\Xi$ & the number of pairs $(e, p) \in {\mc E} \times \Gamma : e \in p$;\\
	$\lambda(t)$ & the total flow entering the network in the origin $o$ at time $t$ {(throughput)};\\
	${\mc S}_{\lambda(t)}$ & The simplex of a probability vector over $\Gamma$ (see \eqref{def:Slambda});\\
	$\beta$ & the fixed noise parameter;\\
	$\eta$ & the update rate of the path preferences; \\
	$L(w)$ & the Lipschitz constant of a function $w$;\\
	$\tilde L$ & the common Lipschitz constant to all the functions belong to $X$ (see \eqref{eq:X});\\
	$|.|$ & cardinality of a set, e.g.,  $|B|$ is the cardinality of set $B$;\\
	$\wedge$ & minimum operator, e.g.,  $a \wedge b = \min\{a,b\}$.	
\end{tabular}
\end{notation}

\section{Model description}\label{sec:2}

We describe the flow dynamics over a \emph{network} of possible paths that
the agents can choose to traverse within a time interval $[0,T]$,  where $T>0$ is the \emph{final horizon}.


\subsection{Network characteristics}

The network is a directed multi-graph $\mathcal{G=(V, E)}$, 
where: $\mathcal{V}$ is a finite set of vertices, generically denoted by $v$, 
and $\mathcal{E}$ is a finite set of directed edges, generically denoted by $e = (\nu_e,\kappa_e)$ 
being $\nu_e$ the tail vertex of $e$ and $\kappa_e\neq \nu_e$ the head vertex.

The {set} $\mathcal{V}$ includes two special vertices, the \emph{orign} $o$ and the \emph{destination} $d$, 
where the agents {enter and leave the network, respectively.
Each edge $e\in\mc E$ is characterized {by three finite parameters}: its \emph{length} $\ell_{e}$; its flow \emph{capacity} $C_e$, 
		expressing the maximum number of agents that can enter in $e$ per unit of time; and {\emph{maximum mass}}~$\rho_{\max}$ denoting 
		the maximum number of agents 
		that can be present at the same time in $e$. We assume $\rho_{\max}$ be the same for each $e \in \mc E$.\\
		An (oriented) \emph{path} from a vertex $v_0$ to a vertex $v_{r}$ is an ordered set of $r$ adjacent edges $p=(e_1, e_2, \ldots, e_{r})$ 
		such that $\nu_{e_1}=v_0$, $\kappa_{e_r}=v_{r}$, $v_s=\kappa_{e_{s}}=\nu_{e_{s+1}}$ for $1\le s\le r-1$, 
		and no vertex is visited twice, i.e., $v_{l}\ne v_s$ for all $0\le l<s\le r$, 
		except possibly for $v_0=v_r$, in which case the path is referred to as a \emph{cycle}.
		A vertex $v_j$ is said to be \emph{reachable} from another vertex $v_k$ if there exists at least a path from $v_k$ to $v_j$.
		
		In particular, we hold the following assumptions on the multi-graph~$\mathcal{G}$:
		\begin{itemize}
			\item $\mathcal{G}$ contains no cycles; 
			\item any vertex in $\mathcal{V}$ can be reached from the origin vertex $o$ 
			and the destination vertex $d$ is reachable from any vertex in $\mathcal{V}$.
		\end{itemize}

		We denote by $\Gamma$ the set of all the paths {$p$} from $o$ to $d$. We denote by {$A$  the $|\mc E| \times |\Gamma|$} \emph{edge-path incidence matrix} with entries 
		\begin{equation}\label{incidencematrix}
		A_{ep}=
		\begin{cases}
		1 & \text{if} \quad e \in p, \\
		0  & \text{otherwise}.
		\end{cases} 
		\end{equation}
		and by
		$$\Xi=\sum_{e \in \mc E}\sum_{p \in \Gamma}A_{ep},\quad \text{with} \quad \vert \mc E\vert \leq \Xi\leq \vert \mc E \vert \times \vert \Gamma\vert,$$
		the number of the elements equal to $1$ of the matrix $A$, {that is, the number of pairs edge-path $(e,p)\in\mc E\times\Gamma$ such that $e\in p$}.
		
		For every path $p\in\Gamma$ and edge $e\in p$, we define two functions
		\begin{equation*}
		\rho^e_p:[0, T]\to [0, \rho_{\max}], \qquad f^e_p:[0, T]\to [0, C_e],
		\end{equation*}
		which denote the current mass and current flow of agents following path $p$, respectively, present and leaving the edge $e$ at  at each time
		instant $t \in [0, T]$. We let
		\begin{equation}
		\rho(t) := \{\rho^e_p(t): e \in p,\, p\in \Gamma\} \in \RR^{\Xi},\qquad f(t) := \{f^e_p(t) : e \in p,\, p\in \Gamma \} \in \RR^{\Xi},
		\end{equation} 
		be the vectors of masses and flows, respectively.

		\bigskip

		{In order to simplify notations and statements, in this paper} we consider a  graph~$\mathcal{G}$ on which agents have only three possible paths to reach~$d$ starting from~$o$ 
		(see Figure~\ref{graphtopology}).
		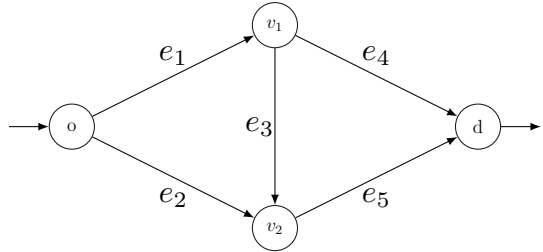
\begin{figure}[thpb]
			\centering
			\begin{tikzpicture}
			[scale=1.35,auto=left,every node/.style={circle,draw=black!90,scale=.6,fill=white!,minimum width=1cm}]
			\node (n1) at (0,0){\Large{o}};  
			\node (n2) at (2,1){\Large{$v_1$}}; 
			\node (n3) at (2,-1){\Large{$v_2$}}; 
			\node (n4) at (4,0){\Large{d}}; 
			\node [scale=0.8, auto=center,fill=none,draw=none] (n0) at (-0.8,0){};
			\node [scale=0.8, auto=center,fill=none,draw=none] (n5) at (4.8,0){};
			\foreach \from/\to in
			{n0/n1,n1/n2,n1/n3,n2/n3,n2/n4,n3/n4,n4/n5}
			\draw [-latex, right] (\from) to (\to); 
			\node [scale=2,fill=none,draw=none] (n5) at (1,0.7){$e_1$};  
			\node [scale=2,fill=none,draw=none] (n6) at (1,-0.7){$e_2$}; 
			\node [scale=2,fill=none,draw=none] (n7) at (1.85,0){$e_3$}; 
			\node [scale=2,fill=none,draw=none] (n8) at (3,0.7){$e_4$};  
			\node [scale=2,fill=none,draw=none] (n9) at (3,-0.7){$e_5$}; 
			\node [scale=1,fill=none,draw=none] (n11) at (4.1,0){};
			\node [scale=1,fill=none,draw=none] (n12) at (2,1.1){};
			\node [scale=1,fill=none,draw=none] (n13) at (2,-1.1){}; 
			\end{tikzpicture}  
			\caption{\label{graphtopology} The graph topology used in the paper.} 
		\end{figure}
		Accordingly, the set of paths is $\Gamma=\{p_1, p_2, p_3\}$, where $p_1=(e_1,e_4)$, $p_2=~(e_2, e_5)$, $p_3=(e_1, e_3, e_5)$. {However, all the results obtained in the next sections can be proved for more general networks, still satisfying the assumptions i) and ii) above.}
		
		
		\subsection{Agents' dynamics and costs}\label{agentdyn}
		
		We assume that the agents are indistinguishable. 
		Each agent enters  the network~$\mathcal{G}$ by the origin vertex, {chooses a path $p \in \Gamma$,} travels through $\mathcal{G}$ along $p$, 
		and finally leaves the network from the destination vertex. 
		
		We let
		{$\lambda:[0,T]\to[0,+\infty[$}
		be {a given} function describing the \emph{throughput} of the agents, i.e., $\lambda(t)$ is the total flow {of agents entering} the network in the origin $o$ at time~$t$.
		In addition, we let $\theta_e \in [0, \ell_e]$ be the state of the generic agent over an edge $e \in \mc E$.
		The value $\theta_e(s)$ describes the position of the agent at time $s$ from the tail of $e$, i.e., $\theta_e(s)=0$ means that the agent is in $\nu_e$, while $\theta_e(s)=\ell_e$ means that the agent is in $\kappa_e$ 
		and hence it is inside the edge $e$ as long as $0\leq\theta_e(s)\leq \ell_e$. 
		We stress that $\theta_e(s)$  describes the state of an {hypothetical} agent assumed to be in $\nu_e$ at time~$t$, 
		independently of the fact {whether there is actually someone present at $\nu_e$ at that time}.
		
		The controlled dynamics in any edge  $e \in \mc E$ of an agent who entered the edge at time $t\in[0,T]$ is:
		\begin{equation}\label{eq:thetau}
		\begin{cases}
		\dot{\theta}_e(s)=u^e(s), &s\in ]t,T],\\
		\theta_e(t)=0,
		\end{cases}
		\end{equation}
		where the control, $s\mapsto u^e(s)$, 
		is  measurable and integrable, namely $u^e \in L^1(0,T)$.

Each agent traversing an edge $e$ at a given time $t$,
aims at minimizing a cost that takes into account: 
i) the possible hassle of running in the edge to reach $d$ on time; 
ii) the pain of being entrapped in a highly congested edge; 
iii) the disappointment of not being able to reach $d$ by the final horizon~$T$.
We model this cost analytically as 
\begin{align}\label{eqcost}
J_e(t, u^e)=&\int_t^T\chi_{\{0\leq\theta_e(s)\leq \ell_{e}\}}\left(\frac{(u^e(s))^2}{2}+\varphi_e\left(\sum_{{\hat p\in\Gamma|e\in\hat p}}\rho^e_{\hat{p}}(s)\right)\right)\,ds\nonumber \\  + 
&\chi_{\{0\leq \theta_e(T)< \ell_e\}}\alpha\sum_{j\in p_e} \ell_{j},
\end{align}
where  $\chi$ is the characteristic function
	\[
	\chi_{\{0\leq\theta_e(s)\leq \ell_{e}\}}=
	\left\{
	\begin{array}{ll}
	\displaystyle
	1&\mbox{if } 0\le\theta_e(s)\le\ell_e,\\
	\displaystyle
	0&\mbox{otherwise},
	\end{array}
	\right.
	\]
	
	\noindent
	and similarly for $\chi_{\{0\leq \theta_e(T)< \ell_e\}}$;
$\alpha>0$ is a constant parameter representing a cost per unit of length, and  $p_e$ is the shortest path from {the tail} $\nu_e$ to $d$.
The quadratic term inside the integral in~(\ref{eqcost}) stands for the cost component i), 
while the other term, characterized  by the congestion function 
$$\varphi_e:[0, \rho_{\max}]\to [0, +\infty[, $$ 
stands for the congestion cost component.
Finally, the last addendum in~(\ref{eqcost})  stands for cost component iii). {In particular, note that, due to the presence of the characteristic functions, the integral part is paid as long as the agent stays on the edge $e$.  The cost outside the integral acts as follows: 1) if at the final horizon $T$ the agent is still in between the edge (not reached the head $\kappa_e$ yet), then the final paid cost is the minimum distance in the graph from the tail $\nu_e$ of the actual edge to the destination $d$; 2) if at the final horizon $T$ the agent is at the head of the edge $\kappa_e$ (i.e. it has already traversed the whole edge), then the corresponding paid cost with respect to the actual edge $e$ is zero. Anyway it will be paid as the minimum distance in the graph from the head vertex $\kappa_e$ to the destination $d$ just by interpreting that head as the tail $\nu_{e'}$ of any other subsequent edge $e'$ hypothetically entered by the agent at time $T$.}

Throughout this paper we will assume the following basic assumptions to hold on the agents' behavior:
\begin{Assumptions}\label{ass:H123}~
	\begin{enumerate}
		\item\label{ass:H1} The throughput $\lambda$ is $C^1([0,T])$ and {$\lambda(t)>0$ for all $t\in[0,T]$. In particular, this implies that there exist $0<\underline\lambda\le\overline\lambda<+\infty$ such that $\underline\lambda\le\lambda(t)\le\overline\lambda$ for all $t\in[0,T]$.}
		\item\label{ass:H2} The initial mass of agents is null, i.e., $\rho(0)=0$.
		\item\label{ass:H3} {For every $e\in\mc E$, the congestion cost function $\varphi_e$ is 
			Lipschitz continuous. Moreover it only depends on the masses $\rho^e_p$ and not on the state variable $\theta_e$.}
		\item\label{ass:H4} The network edges' maximum mass is such that  $\rho_{\max} > \overline\lambda T \geq \int_{0}^{T} \lambda(s) ds$ {and the flow capacity $C_e > \overline\lambda,\, \forall e \in \mc E$, i.e.,
			neither the mass capacity nor the flow capacity of the edges can impede the agents' movements even in the worst case scenario.} 
		\item\label{ass:H5} When more than one optimal control is available, agents choose the smallest one.
		
		\item\label{ass:H6} {Agents have a bounded rationality in the sense that, even when they access to the full available information, 
			the cognitive limitations of their minds, and the finite amount of time they have  prevent them from using the pieces of information to their full extent when making their decisions.}
	\end{enumerate}
\end{Assumptions}
We remark that
Assumption~\ref{ass:H123}.\ref{ass:H2} means that no one is around the network at $t=0$, while
Assumption~\ref{ass:H123}.\ref{ass:H3} implies that all agents in the same edge at the same instant equally suffer the same congestion. Moreover, Assumptions~\ref{ass:H123}.\ref{ass:H1}--\ref{ass:H123}.\ref{ass:H3} imply the boundedness of $\varphi_e$, for all $e \in \mc E$.\\
The simplifying Assumption~\ref{ass:H123}.\ref{ass:H4} will be partially dropped and discussed in the Section~\ref{sec:4}.\\
	Assumptions~\ref{ass:H123}.\ref{ass:H5} and \ref{ass:H123}.\ref{ass:H6} models the human behavior of the agents. Assumptions~\ref{ass:H123}.\ref{ass:H5} implies that agents, when they can choose, prefer to consume less energy than more,
	e.g. they prefer to {move} slower than faster. {In particular, this is implemented in formula (\ref{eq:optimal_controls}), and some other consideration on flow density may also justify it.}  Assumptions~\ref{ass:H123}.\ref{ass:H6} understands that agents typically
 have limited capabilities of
	forecasting the evolution of a dynamic system and of optimizing their decisions. 
	The consequence of this assumption are detailed in the rest of this subsection. Specifically, it will used both in the definition 
	of the agents' aggregate path preference and in the computation of the agents flows~\eqref{flusso}. 

We assume that agents entering the network 
have access to the global information about the current congestion status of the network {through the knowledge of the actual mass vector $\rho$. 
	Then, they choose the path to follow on the basis of their appraisal of the costs of the different paths and on the observation  of the decision of the agents that have preceded.}
Next, we formally introduce this concept.

The relative appeal of the different paths to the agents is modeled by a
time-varying nonnegative \emph{(aggregate) path preference} vector $z:[0,T] \rightarrow \RR_{+}^{\vert\Gamma\vert}$, 
whose generic element $z_p(t)$ represents the {flow's density} of agents entering path $p$ {at the origin $o$} at time $t$.
The vector $z$ varies within the simplex
\begin{equation}\label{def:Slambda}
\mc{S}_{\lambda(t)}=\left\{z\in\RR_{+}^{\vert\Gamma\vert}:\,\sum_{p \in \Gamma }z_p(t)=\lambda(t)\right\},
\end{equation}
where we recall that by $\lambda(t)$ we denote the agents' throughput at time $t$.

The path preference~$z(t)$ evolves over time as a function of the appraisal of the costs that the agents would pay 
along the different paths. The agents assess these costs in terms of the controls that they would implement 
and assuming known the congestion level described by $\rho$.  
Specifically, the cost of each path $p \in \Gamma$ at time $t$ is:
\begin{equation}\label{costoinp}
J^p(t)=\sum_{e \in \mc E: e \in p} J_e(t_e^p(t),u^e_p) ,
\end{equation}
{where, for every $e \in p$,  $u^e_p \in L^1(0,T)$ is the optimal controls implemented  along the edges by an agents who is in the path $p$ (these controls are discussed in the following subsection); $t_e^p(t)$ is the time instant in which an agent, arriving in $t$ in the origin $o$ and following the path~$p$, reaches $\nu_e$ {using the controls $u^e_p$}.{ We write  $t_e^p(t)= \infty$ if an agent does not reach $e$ within $T$ and we define $J_e(\infty,u^e_p) = 0$.} 
	
	We also assume that information on the congestion of the network provided to the agents may be inexact, so that 
	they assess a path~$p$ having a minimum cost 
	with probability $e^{-\beta J^p(t)}/\sum_{\hat p \in \Gamma}e^{-\beta J^{\hat p}(t)}$, where $\beta >0$ is a fixed noise parameter.  
	Hence, the fraction of agents entering the network at time $t$ that would consider a path $p$ having minimum cost is
	\begin{equation}\label{bestresponse}
	F^p_{\beta}(t)
	=\lambda(t)\frac{e^{-\beta J^p(t)}}{\sum_{\hat p \in \Gamma}e^{-\beta J^{\hat p}(t)}}.
	\end{equation}
	
	\noindent
	{Note that,  when $\beta$ tends to~0, then
		$F^p_{\beta}(t)$ tends to $\lambda(t)/|\Gamma|$}, that is, agents consider all the paths equivalent.
	Differently, when $\beta$ tends to infinite 
	the agents have the possibility of {surely} determining the exact costs of the paths and indeed $F^p_{\beta}(t)$ tends to $0$ for all $p$, 
	except for the path minimum cost,  for which it tends to $\lambda(t)$.
	\\
	Hereinafter, we denote by $F_{\beta}(t)$ the vector $\{F^p_{\beta}(t): p \in \Gamma\}$ and by $J(t)=\{J^p(t): p \in \Gamma\}$ the vector of costs on all the paths $p \in \Gamma$.
	
	Agents make their final decision on the path to choose comparing the 
	value of $F_{\beta}(t)$ with the choice of the agents that have preceded them. Specifically, we assume that they correct the difference 
	$z(t) - F_{\beta}(t)$ with a proportional control, as  described by the following equation:
	\begin{equation}\label{evolpi}
	\dot{z}(t) - \dot{F}_{\beta}(t) = - \eta\Big(z(t) - F_{\beta}(t)\Big),\ \ z(0)=z_0,
	\end{equation}
	where, the parameter $\eta>0$ can be interpreted as the rate at which the 
	path preferences are updated. 
	In other words, equation~\eqref{evolpi} says that the bounded rationality of the agents makes them, on the one
	side, like the idea to split as indicated by $F_{\beta}$; on the other side, prefer not to stray from previous agents' decisions.
	We remark that the dynamics described by~\eqref{evolpi} makes $z(t)$ satisfies constraint \eqref{def:Slambda} for all $t \in]0,T]$, whenever the same happens for $z_0$.
	
	\begin{remark}
		Equation~\eqref{evolpi} can be seen as a generalization of the so called \emph{noisy best response dynamics} (see e.g., \cite{Como,ComoMagg}) and such generalization is needed because of the non-constancy of $\lambda$.  While with the noisy best response dynamics, the agents update their path preference comparing the difference between the noisy best response function and their current path preference, in \eqref{evolpi} the agents acts in a way to control the error between the answer to the global information about the actual congestion status and the path preferences of agents who previously entered the network.
		Another possible generalization of the noisy best response dynamics, when $\lambda$ varies over time, is the one given in \cite{CDC}. 
	\end{remark}
The path preference $z$ turns then useful, as in \cite{ComoMagg}, to define, for every $t\in[0,T]$ the \textit{local decision function} $G{[t]}:~\mc {S}_{\lambda(t)}\to\RR_+^{\Xi}$, 
which characterizes the fractions of agents choosing each outward directed edge $e \in p,\, p\in \Gamma$ when traversing a non destination vertex $v$. Actually, in this paper, we are interesting only on the first three component of this functions, $(e_1,p_1), (e_1,p_3), (e_2,p_2)$, which are relative to the two edges $e_1, e_2$ outgoing from the origin $o$ (see Figure~\ref{graphtopology}). We restrict our attention to these three components since once the path is chosen in the origin, in the following non-destination vertices the agents get split according such a {choice}.\\
Hence, we define the first three component of $G{[t]}$ and fix the others equal to zero as follows:
\begin{equation}\label{localchoice}
G{[t]}^{e}_p(z)=
\begin{cases}
\displaystyle\frac{z_p}{\sum_{\hat{p} \in \Gamma} z_{\hat{p}}} & \text{for}\quad e\in \{e_1,e_2\},\, p\ni e,
\\
\,
0 &  \text{for}\quad e \in \{e_3, e_4, e_5\},\, p\ni e.
\end{cases}
\end{equation}
Note that in \eqref{localchoice}, for every $t\in[0,T]$ and for every $z\in\mc{S}_{\lambda(t)}$, it is $\sum_{\hat{p} \in \Gamma}z_{\hat{p}}=\lambda(t){\ge\underline\lambda} >0$, because of (\ref{def:Slambda}) and Assumption~\ref{ass:H123}.\ref{ass:H1}. Hence, for every $t\in[0,T]$, $G[t]$ is a continuous function defined over the compact set $S_{{\lambda}(t)} $, and so uniformly continuous.
Definition~\eqref{localchoice} allows to write the equation that describes mass conservation, for every non-destination vertex $v$ and outward directed edge $e \in p,\, p \in \Gamma$ , as:
\begin{equation}\label{sistemaccoppiato}
\dot{\rho}(t)=H(f(t), z(t); t)\,,\qquad \rho(0)=\rho_0,
\end{equation}
where {the flow $t\mapsto f(t)=(f^e_p(t))^e_p\in\left(\prod_{e\in p}[0,C_e]\right)_p$ is defined next, $t\mapsto z(t)=(z_p(t))_p\in\mc{S}_{\lambda(t)}$ is the solution of (\ref{evolpi}), and $H:\prod_{e\in p}[0,C_e]\times \mc{S}_{\lambda(t)}\to \RR^\Xi$  
	is defined, for every $t\in[0,T]$, by
	\begin{equation}\label{H}
	H^e_p(f(t), z(t); t):= \bigg(\lambda(t)G[t]^{e}_p(z(t))+ f_p^{prec_p(e)}(t)\bigg)-f_p^e(t)\,,\ \forall\  p\in \Gamma,\ e\in p,
	\end{equation}
	\noindent
	with $prec_p(e)$ the function that returns the edge that precedes $e$ on the path $p$. 
	Each component $f_p^e(t)$ of the flow $f(t)$ represents the outgoing flow from the edge $e$ at time $t$.
{Given Assumption~\ref{ass:H123}.\ref{ass:H6}, agents assess the outgoing flow assuming a constant traverse time $k \in ]0,T]$ for each edge $e \in \mc E$. 
			Specifically, $k$ is what the agents assess as the maximum time such that for any $t \in [T- k, T]$ the optimal control $u^e_p(t)$ is certainly null.
			In other words, for $t \in [T- k, T]$, the agents think that it is not convenient to traverse the edge, as the cost of
			running in the edge to reach $d$ at $T$ is for sure greater than the cost of the disappointment of not being able to reach $d$. 
			{Actually, such a value $k>0$ can be a-priori evaluated by the data of the problem.}}  
		Then, we write the outgoing flows as:
		\begin{subequations}\label{flusso} 
			\begin{align}
			f_p^e(t) &= 
			\begin{cases}
			0 &\, \text{if}\, t\in [0,k],\\
			\lambda(t-k)G[t-k]^{e}_p(z(t-k))sign(u_p^e[t-k]) &\, \text{if}\,  t\in [k,T],\\ 
			 \text{for}\,	e\in \{e_1,e_2\},\, p\ni e,
			\end{cases}\\
			f_p^e(t)&=\begin{cases}
			0 &\, \text{if}\, t\in [0,k],\\
			f_p^{prec_p(e)}(t-k)sign(u_p^e[t-k]) &\, \text{if}\,  t\in [k,T],\\
		\text{for}\,	e\in \{e_3,e_4,e_5\},\, p\ni e,
			\end{cases} 
			\end{align}
		\end{subequations}
		where $u^e_p[t-k]\ge0$ 
		is the constant optimal control implemented by an agent that, following path $p$, enters the edge $e$ at time $t-k$, 
		{and $sign(\xi)=1$ if $\xi>0$ and $sign(\xi)=0$ if $\xi=0$.}

{
\begin{remark}
\label{rmrk:flows}
Conditions \eqref{flusso}, coherently with Assumption~\ref{ass:H123}.\ref{ass:H6}, model the outgoing flow $f^e_p(t)$ as possibly
		estimated by an agent entering $e$ at time $t-k$
		that assumes that all the other agents that are currently present on $e$ and that are following the same path $p$, 
		are implementing the same controls $u^e_p[t-k]$, as itself. Of course, a more precise formulation of them should consider the actual value of the control (and not only their sign) and estimate the real traverse time (something similar in this direction is made in \cite{bafama}). Similarly, the mass $\rho$ that satisfies \eqref{sistemaccoppiato} may be more precisely defined in order to represent the real dynamics of the agents. Anyway, such estimated flows and mass evolution may be also seen as an approximation for the elaboration in real time of the information that a possible network manager has to implement in order to send them to the agents. The study of the real discrepancy of such estimated flows and mass evolution form the actual ones may be the subject of future works.

However, note that the estimated flows $f^e_p$, when implemented in ~\eqref{sistemaccoppiato}, make the principle of mass concentration satisfied that is, for example when $\rho_0\equiv 0$, the actual total mass present in the networks is the mass entered through the origin:

\[
\int_0^t\rho(s)ds=\int_0^t\lambda(s)ds\ \forall\ t\in[0,T].
\]

\noindent
Moreover, by \eqref{sistemaccoppiato}--\eqref{flusso}, and by Assumption~\ref{ass:H123}.\ref{ass:H1}, we have that any solution $\rho$ of (\ref{sistemaccoppiato}) is Lipschitz continuous with Lispchitz constant $L=3\overline\lambda$, independently on the optimal control $u$, on the initial value $\rho_0$, and the costs $J$.

	
Finally, let us observe that \eqref{sistemaccoppiato}--\eqref{flusso}
	do not preclude the possibility that agents accumulate at the beginning of an edge~$e$, i.e., on the vertex $\nu_{e}$. This situation may occur, when the optimal control is ${u}^e_p = 0$, since  the corresponding outflow $f^e_p=0$.
\end{remark}
}
	
\subsection{Value functions and optimal controls}\label{VfOC}

	Given a vector mass concentration $\rho(\cdot)$, for each $p\in\Gamma$, $e\in p$ and $t\in[0,T]$, we define the following quantities, representing the optimum that an agent, following path $p$ and entering edge $e$ at time $t$, may get $\forall\,p \in \Gamma$: 
	\begin{equation}\label{Valuefunc1}
	V^e_p(t)=\begin{cases}
	\inf\limits_{\substack{u^e_p \in L^1(0,T)}}
	\Biggl\{ \displaystyle\int_t^{T\wedge\tau} 
	\Biggl(\frac{(u^e_p(s))^2}{2}+\varphi_e\Biggl(\sum_{{\hat p\in\Gamma|e\in\hat p}}\rho^e_{\hat{p}}(s)\Biggr)\Biggr)\,ds+ {\mc F}^e_p(T\wedge \tau)\Biggr\}\\  \text{if }  e \in p\setminus \{last(p)\}, \\ 
	\inf\limits_{\substack{u^e_p \in L^1(0,T)}} \{J_e(t, u^e_p)\} \quad \text{\text{if }}  e = last(p), 
	\end{cases} 
	\end{equation}
	where
	$\tau$ is the first exit time from the closed interval $[0,\ell_e]$, $last(p)$ is a function 
	that returns the last edge of a path $p$ and ${\mc F}^e_p(T\wedge \tau)$ is given by
	\begin{equation}\label{exitcost}
	{\mc F}^e_p(T\wedge \tau)=
	\begin{cases}
	V_p^{succ_p(e)}(\tau) & \text{if } \tau<T,\\
	\displaystyle\alpha\sum_{j\in p_e} \ell_{j} & \text{if } \tau>T,\\
	\displaystyle\min\Bigl\{\alpha\sum_{j\in p_e} \ell_{j},  V_p^{succ_p(e)}(\tau)\Bigr\} & \text{if } \tau=T,
	\end{cases}
	\end{equation}
	with $succ_p(e)$ the function which returns the edge that follows $e$ on path $p$, for $e \in p \setminus \{last(p)\}$.
	%
	
	\noindent{The quantities in (\ref{Valuefunc1}) are recursively and backwardly defined, starting from the ones corresponding to the last edges ending in the destination vertex $d$. We call them, with a little abuse of terminology, \emph{value functions}.
		Note that such a recursive definition is valid as the absence of oriented cycles in the network $\mc G$ prevents self-referring.}
	{The value functions will turn useful in the next section, where we identify a mean field equilibrium. 
		There, instead of considering the standard Hamilton-Jacobi-Bellman equations,
		we will write, as in \cite{bafama}, equivalent conditions in terms of the value functions due to the presence of a discontinuous final cost.}
	\noindent
	{The value functions (\ref{Valuefunc1})} do not depend on the position $\theta_e$ of the agents on the edge {$e\in p$},
	{because, as we are going to show, the optimal behavior of the agents is, for any traversed edge, to implement a constant control $u^e_p\ge0$ chosen when they enter the edges. The main reason for that is the fact that the congestion functions $\varphi_e$ depend on the total mass actually present in the edge and not on the state position of the single agent. Indeed,} consider an agent that, {in an edge $e$}, moves from the {tail} $\nu_{e}$ at time~$t'$  and reaches the vertex $\kappa_e$ at time~$t''$. {Moreover, as we are going to do in the next section, 
		we can suppose the mass concentration $\rho$ as given.}
	The component {$\int_{t'}^{t''} \varphi_{e}(\sum_{\hat p \in \Gamma |e \in \hat p}\rho^{e}_{\hat p}(s))ds$} of the cost~\eqref{eqcost} {can be then assumed as given, whenever the agent in $\nu_e$ at time $t'$ decides to reach $\kappa_e$ at time $t''$.}
	
	\medskip
	Let us now enumerate some facts that, under our hypotheses, hold for the optimal behavior of the agents.
	\begin{enumerate}[label=\roman*)]
\item {When $t''$ is chosen {(which means that the agent has decided to traverse the edge)}, the agent has only to minimize  the component\\ $\frac{1}{2}\int_{t'}^{t''} (u^e_p(s))^2ds$ of the cost $J_e$ in \eqref{Valuefunc1}, and this happens when the control is chosen constant and equal to the constant value $u^e_p=\frac{\ell_{e}}{t''-t'}$. Also note that, with such a choice, in (\ref{Valuefunc1}), it is $\tau=t''$.}
	\item {The previous point i) also} excludes the possibility that an optimally
behaving agent remains at $\nu_{e}$ (i.e, chooses $u^e_p = 0$) for a positive time interval and {then} moves later{; or, similarly, that it stops and stay still in a intermediate point of the edge for a positive time interval;} or that 
the agent goes back and forth along edge $e$. Hence, an optimal control $u^e_p$ is always {constant and} non-negative. 
\item {From the previous points i)--ii), similarly arguing as in \cite{bafama}, we get that optimally behaving} agents cannot accumulate on points strictly internal to the edge, and moreover they also cannot get over each other along the edge because it is impossible that two optimally behaving agents, moving from $\nu_e$ at time $t'_1<t'_2$ respectively, reaches $\kappa_e$ at time $t''_2<t''_1$, respectively. {These facts come from dynamic programming arguments, taking account that any control which is not constant when crossing the edge cannot be otpimal.}
{In particular, this also implies that whenever at time $t$ an optimal choice is $u_p^e=0$ (i.e. to not move) for which the arrival time is $+\infty$, then $u^e_p=0$ will be the unique optimal control for all subsequent instants $t'\ge t$ and hence there will be no controls' multiplicity from this $t$ onwards.}
\item For an agent in $\nu_e$ such that $\kappa_e=d$ (i.e. it stands on the tail of the last edge of the chosen path $p$), it is certainly not optimal to reach $d$ before $T$ and wait there for a positive time length as, in any case, that agent would pay the congestion costs in~$d$ for this interval {(see the cost (\ref{eqcost})}.
\item The following situation is instead possible only for $t''=T$: two optimally behaving agents, moving from $\nu_e$ at time $t'_1<t'_2$ respectively, reaches $\kappa_e$ at the same time $t''$. {Indeed, since the optimal control is necessarily constant (point i)), then any agent that at the time $t$ starts to traverse an edge $e$ as part of a path $p$, has only to optimally choose the arrival time $\tau$ to the vertex of the edge and implement the constant control $u^e_p\equiv \ell_e/(\tau-t)$ (see the terms minimized over $\tau$ in (\ref{eq:V2}), (\ref{eq:V11}), (\ref{eq:V22})). Hence if $\tau=t''<T$, then being $\tau=t''$ a minimizing value internal to $]t,T[$, differentiating and imposing the derivatives equal to zero, we get a contradiction. In \cite{bafama} (Appendix A point 1) the case when the value functions $V$ are not derivable is also treated.}

Note that in this case an accumulation of agents (Dirac mass) may appear in $\kappa_e$, but the time $t''=T$ is the final horizon and hence the game is immediately over and that Dirac mass does not flow.
\item What is instead formally possible is that  for an agent moving from $\nu_e$ at time $t'$,  the choices of reaching $\kappa_e$ at two times $t_1''<t_2''$ are both optimal. In this case, for similar considerations as before, only agents entering the edge at time $t'$ may reach $\kappa_e$ at a time $t''\in[t''_1,t''_2]$. Hence, in the interval $[t''_1,t''_2]$, actually no density of agents arrives and by virtue of similar reasonings to those made in \cite{bafama} (Appendix A point 2) we can
assume, without restriction, that the agents moving from $\nu_e$ at time $t'$ all arrive in $\kappa_e$ at time $t''_2$. More generally, in accordance with Assumption~\ref{ass:H123}.\ref{ass:H5}, for every $t\in[0,T]$, for every edge $e$ and path $p$ containing $e$, we define 
\begin{equation*}
\tau_{e,p}^*(t)=\max\left\{\tau\in]t,T]: u_p^e\equiv\frac{\ell_e}{\tau-t}\ \mbox{is optimal }\right\}
\end{equation*}
\noindent
and then, without restriction, we assume that the optimal control implemented by an agent that starts to traverse the edge $e$ as part of the path $p$ is
\begin{equation}
\label{eq:optimal_controls}
u^e_p\equiv\frac{\ell_e}{\tau^*_{e,p}(t)-t},
\end{equation}
\noindent
when $\left\{\tau\in]t,T]: u_p^e\equiv\frac{\ell_e}{\tau-t}\ \mbox{is optimal }\right\}\neq\emptyset$, and $u^e_p\equiv0$ otherwise (which may corresponds to $\tau^*_{e,p}(t)=+\infty$.
\end{enumerate}

\begin{remark}
\label{rmrk:arrival_time}
By the previous points i)--vi), the function $t\mapsto\tau^*_{e,p}(t)$, whenever it is finite, is increasing. Hence it is continuous almost everywhere and moreover, if $t$ is a continuity point, then $\tau^*_{e,p}(t)$ is the unique possible optimal arrival time.
\end{remark}

{Hereinafter, we denote by $$u=\{u^e_p : e \in p, \, p \in \Gamma,\ u_p^e\ge0\},$$
the controls' vector, and by $u_p^e[t]$ we will denote the optimal constant control chosen by an agent that stands in $\nu_e$ at time $t$ when following the path $p$. Moreover, we do not display the argument $\sum_{\hat{p} \in \Gamma |e \in \hat p}\rho^{e}_{\hat{p}}$ of $\varphi_{e}$, whenever it is not strictly necessary.}

Consider now an agent standing at $\nu_{e}$ {at time $t<T$, and hence at $\theta_e(t)=0$, where $\kappa_e=d$, i.e (looking to the Figure~\ref{graphtopology} )} for the pairs $(e,p) \in \{(e_4,p_1), (e_5,p_2), (e_5,p_3)\}$.
It has two possible choices: either staying at $\nu_{e}$ indefinitely or moving to reach $\kappa_{e}=d$ exactly at time~$T$.
Accordingly, the candidate {constant} optimal controls {to be chosen at the time $t$} are
\begin{equation}\label{oc2}
u^{e}_{p,1}{[t]\equiv}0, \qquad u^{e}_{p,2}{[t]\equiv}\frac{\ell_{e}}{T-t}.
\end{equation} 
Hence, given the cost functional (\ref{eqcost}), we derive
\begin{equation}\label{eq:V1}
V^{e}_p(t)=\min\left\{ \alpha\ell_{e},\,\frac{1}{2}\frac{(\ell_{e})^{2}}{T-t}\right\} +\int_{t}^{T}\varphi_{e}\,ds.
\end{equation}
An agent standing at $\nu_{e_3}$ at time $t\in[0,T]$ has two possible
choices: staying in $\nu_{e_3}$ or moving to reach $\kappa_{e_3}$ at {some (optimal) instants} $\tau \in ]t, T]$. 
Hence, we obtain that the agent has to choose between the following {two kinds of candidate constant optimal controls}:
\begin{equation}\label{oc1}
u^{e_3}_{p_3,1}[t]\equiv0,\qquad \ u^{e_3}_{p_3,2}[t]\equiv\frac{\ell_{e_3}}{\tau-t},
\end{equation}
whose associated value function is:
\begin{equation}
\label{eq:V2}
V^{e_3}_{p_3}(t)= 
\min\left\{ \alpha\left(\ell_{e_3}+\ell_{e_5}\right)+\int_{t}^{T }\varphi_{e_3}\,ds,\inf_{\tau\in ]t,T]}  \left\{\frac{1}{2}\frac{(\ell_{e_3})^{2}}{\tau-t}+\int_{t}^{\tau }\varphi_{e_3}\,ds+V^{e_5}_{p_3}(\tau)\right\}\right\}
\end{equation}

An agent standing at $\nu_{e_1}$ at time $t$ and following a path $p \in \{p_1,p_3\}$ may choose between staying in $\nu_{e_1}$ or reaching $\kappa_{e_1}$ at a certain $\tau \in ] t,T]$. 
Hence, the candidate {constant} optimal controls are:
\begin{equation}\label{oc4}
u^{e_1}_{p,1}[t]\equiv0,  \qquad  u^{e_1}_{p,2}[t]\equiv\frac{\ell_{e_1}}{\tau-t}
\end{equation}
whose associated value functions are:
\begin{subequations}\label{eq:V11}
	\begin{align}\label{eq:V11a}
\hspace{-1cm}&V^{e_1}_{p_1}(t)=\min\left\{ \alpha \Big( \ell_{e_1}+\ell_{e_4} \Big)+\int_{t}^{T}\varphi_{e_1}\,ds,
	\inf_{\tau\in ]t,T]}\left\{\frac{1}{2}\frac{(\ell_{e_1})^{2}}{\tau -t}+\int_{t}^{\tau }\varphi_{e_1}\,ds+ V_{p_1}^{e_4}(\tau)\right\}\right\}	\\
	\label{eq:V11b}& V^{e_1}_{p3}(t)=\min\Bigg\{\alpha \Big( \ell_{e_1}+\ell_{e_3}+\ell_{e_5}  \Big)+ \int_{t}^{T}\varphi_{e_1}\,ds,\\ 
	&\hspace{6cm}\inf_{\tau\in ]t,T]}\left\{\frac{1}{2}\frac{(\ell_{e_1})^{2}}{\tau -t}+\int_{t}^{\tau }\varphi_{e_1}\,ds+ V_{p_3}^{e_3}(\tau)\right\}\Bigg\}.\nonumber	
	\end{align}
\end{subequations}

Analogous arguments hold for computing $V^{e_2}_{p_2}(t)$ when an agent is standing at $\nu_{e_2}$.
The candidate {constant} optimal controls are
\begin{equation}\label{oc3}
u_{p_2,1}^{e_2}[t]\equiv0, \qquad u_{p_2,2}^{e_2}[t]\equiv\frac{\ell_{e_2}}{\tau-t},
\end{equation}
whose associated value function is:
\begin{equation}\label{eq:V22}
V^{e_2}_{p_2}(t)=\min\Bigg\{\alpha(\ell_{e_2}+\ell_{e_5})+\int_{t}^{T}\varphi_{e_2}\,ds,
\inf_{\tau\in ]t,T]}\left\{\frac{1}{2}\frac{(\ell_{e_2})^{2}}{\tau -t}+\int_{t}^{\tau }\varphi_{e_2}\,ds+ V^{e_5}_{p_2}(\tau)\}\right\}\Bigg\}.
\end{equation}

\begin{remark}
	\label{rmrk:valuefunctions}
	We remark that, the  optimal controls described in (\ref{oc1}), (\ref{oc4}), (\ref{oc3})
	are detected,  along with the {possible arrival time $\tau$, by} the minimization process carried on in  (\ref{eq:V2}), (\ref{eq:V11}), (\ref{eq:V22}).
	Also, when $\rho$ is given, the construction of the optimal controls may be  performed  backwardly, 
	starting from the problem~(\ref{eq:V1}). {Also note that, the minimization processes in $\tau$ are admissible because of the coercivity of the minimizing term when $\tau\to t^+$.} 
\end{remark}

We now give in the following a result of Lipschitz continuity of the value functions defined above that will turn useful in the next section.
\begin{proposition} \label{th:valuefun} Suppose that $\rho$ is given continuous and that Assumptions~\ref{ass:H123} hold. 
	Then, every value function $V^e_p :[0,T] \to \mathbb{R}$, for all $e \in p,\, p\in \Gamma$ defined by (\ref{eq:V1})-(\ref{eq:V22})
	is: Lipschitz continuous, with Lipschitz constant independent of {$\rho$};  
	bounded independently on $\rho$;
	continuous with respect to the mass density $\rho$ (via the congestion functions $\varphi$), 
	i.e, whenever $\rho_n\to\rho$ uniformly, then $V_{p,n}^e\to V_p^e$ uniformly in $[0,T]$.
\end{proposition}
\begin{proof} 
Assumptions~\ref{ass:H123} implies that $\sum_{\hat{p} \in \Gamma | e \in \hat p}\rho^{e}_{\hat{p}}{\le\rho_{\max}}$ is bounded
, {independently from controls, paths and edges}, 
then there exists a positive constant $k_1$ {such that, for every $0\le t_1\le t_2\le T$, it always holds:}
\begin{equation}\label{lip}
\left\vert\int_{t_1}^{t_2} \varphi_e\left(\sum_{\hat{p} \in \Gamma | e \in \hat p}\rho^{e}_{\hat{p}}(s)\right)\,ds\right\vert \leq
\left\Vert\varphi_e\left(\sum_{\hat{p} \in \Gamma | e \in \hat p}\rho^{e}_{\hat{p}}(\cdot)\right)\right\Vert_{\infty }\hspace{-0.2cm}\left\vert t_2
-t_1\right\vert \le k_{1}\left\vert t_2 -t_2\right\vert \leq k_{1}T.
\end{equation}

\noindent
Now, take $e$ as the last edge of the path $p$ (i.e. looking to see Figure~\ref{graphtopology},  $(e,p) \in \{(e_4,p_1), (e_5,p_2), (e_5,p_3)\}$), and consider $V^e_p$ as defined in (\ref{eq:V1}). It is evident that it is of the form
\begin{equation*}
V^e_p(t)=\frac{1}{2}\frac{(\ell_e)^2}{T-t}+\int_t^T\varphi_e ds,
\end{equation*}

\noindent
only if $T-t\ge\ell_e/(2\alpha)$, that is $t\le T-\ell_e/(2\alpha)\le T-h$ with $h>0$ independent on $p$ and its last edge $e$, on $\rho$ and on controls. Using also (\ref{lip}), we then get the Lipschitz continuity of all value functions $V_p^e$ in (\ref{eq:V1}), with the same Lipschitz constant. We also easily get the equiboundedness of those $V_p^e$.

 Proceeding backwards, {let us consider} $V^{e_3}_{p_3}(t)$ given by (\ref{eq:V2}). We concentrate on the term minimized with respect to $\tau\in]t,T]$ in (\ref{eq:V2}). Again, as before (see also Remark \ref{rmrk:valuefunctions}), there exists $h>0$ independent on $\rho$, on controls and on $t\in[0,T]$ such that, for any $t$, whenever $V_{p_3}^{e_3}(t)$ is defined as that minimized term, then the minimizing values $\tau$ belong to $[t+h,T]$ (and $V_{p_3}^{e_3}(t)$ is certainly defined as the other term in the exterior minimization in (\ref{eq:V2}) when $t+h>T$). Hence, for every $t$, we consider the function
\begin{equation*}
\psi^t:[t+h,T]\to\mathbb{R},\ \tau\mapsto \frac{1}{2}\frac{(\ell_{e_3})^2}{\tau-t}+\int_t^\tau\varphi_{e_3}ds+V_{p_3}^{e_5}(\tau).
\end{equation*}
	\noindent
	Note that $\psi^t$ is Lipschitz continuous for every $t$, with Lipschitz constant $M>0$ independent on $t$ and on $\rho$ (because so is $V_{p_3}^{e_5}$ from previous considerations). For $0\le t_1<t_2\le T$, and for $\tau\in[t_2+h,T]$, we get (see also (\ref{lip})), again for $M>0$ independent from all,
\begin{equation*}
\begin{split}	
	\left|\psi^{t_1}(\tau)-\psi^{t_2}(\tau)\right|&\le\frac{1}{2}\left|\frac{(\ell_{e_3})^2}{\tau-t_1}-\frac{(\ell_{e_3})^2}{\tau-t_2}\right|+\int_{t_1}^{t_2}\varphi_{e_3}ds\\
	 &\le\frac{1}{2}\frac{(\ell_{e_3})^2}{h^2}|t_1-t_2|+k_1|t_1-t_2|=M|t_1-t_2|.
	\end{split}
	\end{equation*}
	
	\noindent
	Let $\tau_1,\tau_2$ be two points of minimum for $\psi^{t_1}$ and $\psi^{t_2}$ respectively. We get
	\begin{equation*}
	\psi^{t_1}(\tau_1)-\psi^{t_2}(\tau_2)\le\psi^{t_1}(\tau_2)-\psi^{t_2}(\tau_2)\le M|t_1-t_2|.
\end{equation*}
	\noindent
	If $\tau_1\ge t_2+h$, we then similarly get
	\begin{equation*}
	\psi^{t_2}(\tau_2)-\psi^{t_1}(\tau_1)\le\psi^{t_2}(\tau_1)-\psi^{t_1}(\tau_1)\le M|t_1-t_2|.
	\end{equation*}
	
	\noindent
	If instead, $t_1+h\le\tau_1<t_2+h$, then we get
		\begin{equation*}
	\psi^{t_2}(\tau_2)-\psi^{t_1}(\tau_1)=\psi^{t_2}(\tau_2)\pm\psi^{t_2}(t_2+h)\pm\psi^{t_1}(t_2+h)-\psi^{t_1}(\tau_1)\le 2M|t_1-t_2|.
	\end{equation*}
	
	\noindent
	We then get the Lipschitz continuity of $V_{p_3}^{e_3}$ in (\ref{eq:V2}), with Lipschitz constant independent on $\rho$. 

Arguing similarly, {in a backward manner}, one proves the Lipschitz continuity of the value functions in \eqref{eq:V11} and \eqref{eq:V22}, with Lipschitz constant independent on $\rho$.

{Now, still proceeding backwardly, for a uniformly convergent sequence of mass densities $\rho_n\to\rho$, we easily get that the corresponding value functions in (\ref{eq:V1}) uniformly converge. From this, we obtain that the corresponding value functions in (\ref{eq:V2}) and (\ref{eq:V11}) pointwise converge. But they are also equibounded and equi-Lipschitz and so uniformly converge. We conclude proceeding backwardly i this way.
} 
\end{proof}

%

\begin{remark}\label{noteLip}
	An immediate consequence of Proposition~\ref{th:valuefun} and of Assumption~\ref{ass:H123}.\ref{ass:H1} is that 
	$F^p_{\beta}(t)$ defined in \eqref{bestresponse} is Lipschitz and equi-Lipschitz continuous.
	Indeed,  $F^p_{\beta}(t)$ is built considering the cost $J^p$ in \eqref{costoinp} which is the ``sum" of the value functions $V_p^e$ \eqref{Valuefunc1}  that by Proposition \ref{th:valuefun} are Lipschitz and equi-Lipschitz continuous.
\end{remark}

\section{Existence of a mean field equilibrium}\label{sec:3} 
In this section we prove the existence of a mean field equilibrium for $\rho$ over the considered network $\mathcal{G}$.
Specifically, we proceed as follows.\\
First, we let $L(w)$ be the Lipschitz constant of a function~$w$ and we  choose
as a space to search for a fixed point: 
\begin{equation}\label{eq:X}
X=\left\{w:[0,T]\to[0,\rho_{\max}]:  {L(w)\le\tilde L},\, |w|\le \rho_{\max} \right\}^\Xi,
\end{equation}
the Cartesian product $\Xi$ times of the space of  Lipschitzian functions with Lipschitz constant not greater than $\tilde L$  
and overall bounded by $\rho_{\max}$, where $\tilde L$ is a constant. 
Space $X$ is convex and compact with respect to the uniform topology.\\
Then, fixed the noisy parameter $\beta>0$, we search for a fixed point of the function 
$\psi:~X\to~X$,  with $\rho\mapsto \rho^\prime =\psi(\rho)$ where $\rho^\prime$ is obtained performing the following steps (see diagram in Fig.~\ref{puntofisso}):

\begin{enumerate}[label=\roman*)]
	\item given the mass $\rho$ the optimal control $u$ is derived through (\ref{eq:V1})-(\ref{eq:V22});
	\item the optimal control $u$ is used both to compute the flow vector $f$ through \eqref{flusso} and to obtain the path preference vector $z$ through~\eqref{evolpi} by first computing the vector of costs $J$;
	\item the mass vector~$\rho^\prime$ is derived from~$f$ and~$z$ through~\eqref{sistemaccoppiato} by first computing the vectors~$G$ through~\eqref{localchoice} and~$H$  through~\eqref{H}.
\end{enumerate}
\begin{figure}[thpb]
	\centering
	\begin{tikzpicture}
	[scale=1.4,auto=left,every node/.style={circle,draw=black!90,scale=.5,fill=white,minimum width=.5cm}]
	\node [scale=2, auto=center,fill=none,draw=none] (n0) at (0,0){$\rho$};
	\node [scale=2, auto=center,fill=none,draw=none] (n1) at (1,0){$u$};
	\node [scale=2, auto=center,fill=none,draw=none] (n2) at (2.2,0){$J$};
	\node [scale=2, auto=center,fill=none,draw=none] (n3) at (3.4,0){$z$};
	\node [scale=2, auto=center,fill=none,draw=none] (n4) at (4.5,0){$G(z)$};
	\node [scale=2, auto=center,fill=none,draw=none] (n5) at (6,0){$H(f,z)$};
	\node [scale=2, auto=center,fill=none,draw=none] (n6) at (7.2,0){$\rho'$};
	\node [scale=2, auto=center,fill=none,draw=none] (n7) at (3.4,-0.8){$f$};	
	\foreach \from/\to in
	{n0/n1,n1/n7,n1/n2,n2/n3,n3/n4,n4/n5,n7/n5,n5/n6}
	\draw [-latex, right] (\from) to (\to); 
	\end{tikzpicture} 
	\caption{\label{puntofisso} Fixed point scheme.}
\end{figure}
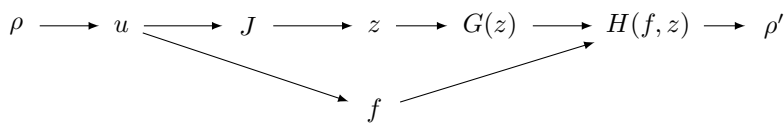
{Note that a suitable constant $\tilde L$ exists such that the function $\psi$ maps $X$ into itself. Indeed, 
note that, by construction, $\psi(\rho)$ must satisfy (\ref{sistemaccoppiato}) and hence, by Remark \ref{rmrk:flows} and  Assumption~\ref{ass:H123}.\ref{ass:H4}, the bound $|\rho|\le\rho_{max}$ is satisfied and, as Lipschitz constant we can take $\tilde L=3\overline\lambda$.
}

{
\begin{definition}\label{def:MFGE}
		Let $\psi$ the function described above. Then a
		{mean field equilibrium} is a total mass $\rho\in X$ that satisfies $\rho = \psi(\rho)$.	
\end{definition}

Now we show that the function $\psi$ is continuous so that Brouwer fixed-point theorem can be applied and a mean field equilibrium exists.}
	
\begin{lemma}\label{lem:2} The function $\psi:X\to X$ is continuous.
\end{lemma} 
\begin{proof} 
We show that for every sequence $\{\rho^n\} \subset X$ and for every $\rho \in X$ such that $\rho^n\to \rho$ {uniformly}, we get $\psi(\rho^n)\to \psi(\rho)$ {uniformly}.
We divide the proof into several steps.

(1)  Consider the value functions $V^{e,n}_p$ and $V^e_p$, for every $e \in p,\, p\in \Gamma$ defined  by (\ref{eq:V1}),(\ref{eq:V2}),(\ref{eq:V11}),(\ref{eq:V22})  
and associated, respectively, to the  choices of masses $\rho^n$ and $\rho$ in the congestion cost vector
$\varphi=\{\varphi_e: e \in \mathcal{E}\}$, where each component $\varphi_e$ has as argument the corresponding component of  $\rho^n$ and $\rho$, respectively. 
By {Proposition \ref{th:valuefun} since $\rho^n\to\rho$ uniformly, then 
$V^{e,n}_p\to V^e_p$ $\forall\,e \in p$ uniformly in $[0,T]$}.
 {

For every fixed $t$, let $u^n[t]$ and $u[t]$ be the corresponding constant optimal controls for traversing at time $t$ a given edge $e$ in a given path $p$ (here not displayed), with the corresponding optimal arrival time $\tau^*_n(t)$, $\tau^*(t)$ (see (\ref{eq:optimal_controls})). By compactness, there exists a real number $u^t$ such that, at least for a subsequence, $u^n[t]\to u^t$. By the convergence of the value functions, and consequently of the minimizing expressions (\ref{eq:V1}),(\ref{eq:V2}),(\ref{eq:V11}),(\ref{eq:V22}), we have that the constant $u^t$ is an optimal constant control for traversing, at time $t$, the edge $e$, as part of the path $p$, with the given limit mass $\rho$. By Remark \ref{rmrk:arrival_time}, if $t$ is a continuity point of $\tau^*(\cdot)$, then the only optimal control for the limit problem is $u[t]\equiv\frac{\ell_e}{\tau^*(t)-t}$, and hence the limit is independent from the subsequence. Again by Remark \ref{rmrk:arrival_time}, we then get that the sequence of optimal control functions $u^n[\cdot]$ almost everywhere converges to the limit optimal control $u[\cdot]$. By the dominated convergence theorem they then converge in $L^1(0,T)$. 
}

\medskip

(2) Consider the functions  $F^{p, n}_{\beta}$ and  $F^p_{\beta}$ for every $p \in \Gamma$ defined by \eqref{bestresponse} and associated, respectively, to the optimal controls $u^n[\cdot]$ and $u[\cdot]$ introduced in the point (1).
By Remark \ref{noteLip} and 
since  $V^{e,n}_p\to V^e_p$ $\forall\,e \in p$ uniformly, it follows that $F^{p,n}_{\beta} \to F^p_{\beta}(t)\, \forall\,p \in \Gamma$ uniformly. Let now $\{z^n\}$ and $z$ be the sequence of path preference vectors and the path preference vector induced, respectively, by $F^n_{\beta}$ and  $F_{\beta}$ through \eqref{evolpi}. Note that  the sequence $z^n$ is equi-bounded and equi-Lipschitz continuous (since $z^n$, $F^n_{\beta}$ and $\dot{F}^n_{\beta}$ are bounded), hence, there exists $\widetilde{z}$ such that, at least along a subsequence, $z^n\to \widetilde{z}$. Now using \eqref{evolpi} for both $z^n$ and $\widetilde{z}$ we get
\begin{subequations}
	\begin{align}
	\label{1eq} z^n(t) &=z^n(0)-\eta \Big(\int_0^t z^n(s)\,ds - \int_0^t F^n_{\beta}(s)\,ds\Big)+ F^n_{\beta}(t) - F^n_{\beta}(0),\\
	\label{2eq} \widetilde{z}(t) &= \widetilde{z}(0) - \eta \Big(\int_0^t \widetilde{z}(s)\,ds - \int_0^t F_{\beta}(s)\,ds\Big)+ F_{\beta}(t) - F_{\beta}(0).
	\end{align}
\end{subequations}
From the above considerations follows that the right hand side of \eqref{1eq} converges to the right hand side of \eqref{2eq}. Hence, by the uniqueness of the solution of \eqref{evolpi} one gets that $\widetilde{z}(t)=z(t)\, \forall\, t \in [0, T]$ and $z^n(t)\to {z(t)}$.  Say that, since the function $G[t]$ is uniformly continuous then $G[t](z^n(t))$ converges to $G[t](z(t))$.	

\medskip

(3) Taking into account the optimal controls {$u^n[\cdot]$ and $u[\cdot]$ introduced in the point (1) such that $u^n[\cdot]\to u[\cdot]$ in $L^1(0,T)$ and almost everywhere,} and given the throughput $\lambda$ for every $t \in [0, T]$, we can compute the corresponding flows $f^n$ and $f$ as in \eqref{flusso}. We now want to prove that $f^n\to f$ {in $L^1(0,T)$} for which it is enough to show that {$sign(u^n[\cdot])\to sign(u[\cdot])$ in $L^1(0,T)$}.\\
By the optimization procedure \eqref{oc2}--\eqref{eq:V22} follows that,
each agent when enters an edge $e$ decides either to stop or 
to keep a constant control strictly grater than zero, which 
allows the agent to reach the other extreme of the edge within time $T$.
Then, any control  $u[\cdot] >0$ is lower bounded by a constant $\frac{\ell_e}{T} > 0$ (for every edge $e$ in a given path $p$).
As a consequence if $u^{n}[\cdot]\to u[\cdot] > 0$, we have $u[\cdot] \geq \frac{\ell_e}{T} > 0$.{ Hence, $sign(u^{n}[\cdot])\to sign(u[\cdot])=1$.\\
Differently, if $u^{n}[\cdot]\to u[\cdot]=0$, by the limit definition follows that from a certain $n$ onward $u^{n}[\cdot] < \frac{\ell_e}{T}$ and hence, by its optimality, $u^{n}[\cdot]=0$ which in turn implies that $sign(u^{n}[\cdot])\to sign(u[\cdot])=0$.  
Therefore we have proven the {almost everywhere convergence of signs from which, by the dominated convergence, their convergence in $L^1(0,T)$, and hence the one of the flows.}}

\medskip

    Then we can compute {(edge by edge)} $\psi(\rho^n)$ and $\psi(\rho)$ integrating the mass conservation \eqref{sistemaccoppiato}:
\begin{subequations}
	\begin{align}
	\label{Im1}\psi(\rho_n(t))&= \rho^n(0)+\int_0^t \left(\lambda(s)G[s](z^n(s))+f^{prec, n}(s)\right)\,ds -\int_0^t f^n(s)\,ds;\\
	\label{Im2}\psi(\rho(t))&= \rho(0)+\int_0^t \left(\lambda(s)G[s](z(s))+f^{prec}(s)\right)-\int_0^t f(s)\,ds.
	\end{align}
\end{subequations}
Now, using all the previous arguments in the points (1), (2) and (3) we get that the right hand side of \eqref{Im1} converges to the right hand side of \eqref{Im2}, 
{from which $\psi(\rho_n(t))\to \psi(\rho(t))$ for every $t\in[0,T]$, and also uniformly, being them equibounded and equi-Lipschitz because belonging to $X$}.
Hence, by Brouwer fixed point theorem, the map $\rho \to\psi(\rho)$ has a fixed point which is the mean field equilibrium.
\end{proof}

\section{Mass-depending bounded controls}\label{sec:4}

In the previous sections we have assumed that the set of admissible values for the controls $u$ was the whole real line 
$\mathbb{R}$, even if, from an optimization argument, the really implemented controls were non-negative and bounded. 
This fact implied that, at least formally, each agent has at disposal any possible values for the control, 
which we recall can be interpreted as scalar velocity, even if the edge is very 
congested. 
From a modeling point of view, this may be not satisfying. 
Hence, here we assume that there is bound on the set of admissible controls, and that such a bound somehow depends on the actual values of the mass concentration $\rho^e$ on the edge $e$, coherently with the feature of our model, where any agent in the edge $e$ at time $t$ suffers the same congestion $\rho^e(t)$. 
Hence, for every edge $e$, we consider a function $U^e:]0,+\infty[\to[0,+\infty[$, such that

i) $U^e$ is continuous and decreasing and strictly positive;

ii) $\lim_{\xi\to0^+}U^e(\xi)=+\infty$, $\lim_{\xi\to+\infty}U^e(\xi)=0$.

\noindent
We then assume that, at any time $t$, an agent in the edge $e$ has at disposal the bounded interval $[0,U^e(\rho^e(t))]$, as admissible values for controls. That is, if in the time interval $[t_1,t_2]$, an agent is in the edge $e$, then it can only use measurable controls such that 
\begin{equation}
\label{eq:constraint}
u(s)\in[0,U^e(\rho^e(s))]\ \mbox{a.e. } s\in[t_1,t_2].
\end{equation}

\noindent
Note that, without loosing generality, we already restrict ourselves to non-negative controls: indeed, also in this case, by an optimization point of view, the use of negative controls (i.e. to move back on the edge) will be certainly not optimal.

We now suppose that the continuous evolution of the mass distribution $t\mapsto\rho(t)$ is given (as in the fixed point procedure). In the previous sections, again by optimization arguments, see \eqref{oc1}--\eqref{eq:V22}, the actual optimization parameter for an agent entering the edge $e$ at time $t$ was just $\tau>t$, the arrival time on the vertex of the edge, and then, when moving was optimal, the optimal control to be implemented was the constant one $u\equiv\ell_e/(\tau-t)$. Hence, for every edge $e$ and every time $t\in[0,T]$, we define
\begin{equation*}
\tau(t,e,\rho^e)=\tau>t\ \mbox{such that } \int_t^\tau U^e(\rho^e(s))ds=\ell_e
\end{equation*}
with the convention that $\tau(t,e,\rho^e)=+\infty$ when such $\tau>t$ does not exist in $[t,T]$. Hence $\tau(t,e,\rho^e)$, when finite, represents the minimal arrival time on $\kappa_e$ for an agent entering the edge  $e$ at time $t$ and using controls satisfying the constraint (\ref{eq:constraint}) in $[t,\tau(t,e,\rho^e)[$, whereas, when it is infinite, it means that there is no possibility to reach $\kappa_e$ by the final time $T$. 

Note that, restricting ourselves to the values $t$ such that $\tau(t,e,\rho^e)<+\infty$, the function $t\mapsto\tau(t,e,\rho^e)$ is strictly increasing. Indeed, if for some $t_1<t_2$ we have $\tau(t_2,e,\rho^e)\le\tau(t_1,e,\rho^e)$, then we would have
\begin{equation*}
\int_{t_1}^{t_2}U^e(\rho^e(s))ds+\int_{\tau(t_2,e,\rho^e)}^{\tau(t_1,e,\rho^e)}U^e(\rho^e(s))ds=0,
\end{equation*}
which is a contradiction due to the strict positivity of $U^e$.

Now, let us note that, even if $\tau\ge\tau(t,e,\rho^e)$, then the corresponding constant velocity $u\equiv\ell_e/(\tau-t)$ of traversing the edge does not necessarily satisfy the constraint (\ref{eq:constraint}). On the other side, we would like to recover, in this constraint case too, many of the results of the previous sections, in particular all the properties of the optimal control (see i)--vi) Section \ref{VfOC}) coming from their constancy when traversing the edge.  To this end, we relax our constrained optimal control problem (with constraint given by (\ref{eq:constraint})) in the following one:

{\it Constraint on the arrival time}: every agent that enters the edge $e$ at time $t\ge0$ optimizes \eqref{oc1}--\eqref{eq:V22} among $\tau\in[\tau(t,e,\rho^e),+\infty]$. That is, it can implement any measurable control (not necessarily satisfying the constraint (\ref{eq:constraint})), provided that it satisfies the constraint on the arrival time on $\kappa_e$: the arrival time must be not less than $\tau(t,e,\rho^e)$.

In order to state such a new mean field game problem with lower bound on the arrival time, instead of starting from  the existence of the function $U$ giving the velocity-constraint (\ref{eq:constraint}), we start from the existence of a given arrival-time-constraint function with suitable properties.

\begin{Assumption}\label{ass:Hsez4}
For every edge $e \in \mc E$ there exists a function
\begin{equation*}
\tau(\cdot,e,\cdot):[0,T]\times C^0([0,T],\mathbb{R}^+)\to[0,+\infty[,\ \ (t,\rho^e)\mapsto\tau(t,e,\rho^e)
\end{equation*}
such that:

a) it is Lipschitz-continuous (with $C^0([0,T],\mathbb{R}^+)$ endowed by the uniform topology);

b) $\tau(t,e,\rho^e)> t\ \forall\ (t,\rho^e)$; 

c) it is strictly increasing in $t$, for every $\rho^e$ fixed;

d) it is strictly increasing in $\rho^e$ for every fixed $t$, that is
\begin{align*}
\rho_1^e\le\rho_2^e\ \mbox{in } [t,\tau(t,e,\rho_1^e)],\ \exists\ s\in[t,\tau(t,e,\rho^e_1)]\ \mbox{such that } & \rho_1^e(s)<\rho_2^e(s) \Longrightarrow \\
& \tau(t,e,\rho^e_1)<\tau(t,e,\rho_2^e).
\end{align*}
\end{Assumption}

Hence, in this setting, the mean field game problem is as the one in the previous sections, with the only difference that in the minimization of the costs \eqref{eqcost}, every agent entering the edge $e$ at time $t\ge0$ implements controls from the set
\begin{equation*}
{\mc U}(t,e,\rho^e)=\left\{u\in L^1(0,T)  : \mbox{ the corresponding arrival time is } \tau\ge\tau(t,e,\rho^e) \right\},
\end{equation*}

\noindent
instead of controls from the whole space $L^1(0,T)$.

In order to apply to this setting all the argumentation and calculations of the previous sections, we have to test the validity of the points i)--vi) of Section \ref{VfOC}, and the Lipschitz continuity of the value functions \eqref{eq:V1}--\eqref{eq:V22} where, in this case, the minimization in $\tau$ are, instead of for $\tau\in]t,T]$, for $\tau\in[\tau(t,e,\rho^e),T]$. In what follows, we tacitly refer to those points.\\

i)--ii) For every $\tau\ge\tau(t,e,\rho^e)$ the constant control $\ell_e/(\tau-t)$ belongs to ${\mathcal U}(t,e,\rho^e)$ with arrival time $\tau$, and hence, for the same reasons it is the minimizing one. Also iii) and iv) come again from the same considerations as in the Section \ref{VfOC}.\\

v) Here, we observe that, since $t\mapsto\tau(t,e,\rho^e)$ is strictly increasing, then, if for $t_1<t_2$ we have the same optimizing arrival time $\tau$, it must be $\tau(t_1,e,\rho^e)<\tau(t_2,e,\rho^e)\le\tau\le T$. Taking $t_1<t'<t_2$, by the points iii), agents starting at time $t'$ must have the same arrival time $\tau$. Hence we get $\tau(t_1,e,\rho^e)<\tau(t',e,\rho^e)<\tau\le T$. If then we assume $\tau<T$, we then get a contradiction because $\tau$, being and interior minimizing point of the costs of agents starting at $t_1$ as well as at $t'$, is a stationary point and we conclude using the first order condition (see \cite{bafama}, where the possible non-differentiability of $V$ is also taken into account).\\

vi) This point is similarly valid in this constrained case.\\

For the Lipschitz continuity of the value functions $V^e_p$ \eqref{eq:V1}--\eqref{eq:V22}, we just observe that now the minimization is for $\tau\in[\tau(t,e,\rho^e),T]$ but, as done in the proof of Proposition \ref{th:valuefun}, the minimizing $\tau$ of the function $\psi^t$ still belongs to $[t+h,T]$.

Finally, we observe that in the definition of the flows \eqref{flusso}, $k$ is defined as a quantity such that to start to traverse the edge at a time after $T-k$ is certainly not convenient. However in that definition it has also the meaning of a (approximately estimated) mean minimal traversing time. Here, in this constrained situation, it would be more precise to take account also of the minimal traversing time due to the constraint $\tau(t,e,\rho^e)$. Hence we define the mean minimal traversing time as
\begin{equation*}
\overline\tau(e,\rho^e)=\frac{1}{T}\int_0^T\left(\tau(s,e,\rho^e)-s\right)ds
\end{equation*}
and replace $k$ in \eqref{flusso} by $\tilde k=\max\{k,\overline\tau(e,\rho^e)\}$ (which is sufficiently less than $T$ if $T$ is large, otherwise we can suitably cut it). Note that, by our hypotheses, the function $\rho^e\mapsto\overline\tau(e,\rho^e)$ is continuous with respect to the uniform convergence and hence we can still apply all the fixed point machinery as in the previous section.

\begin{remark}
	The constrained case here discussed may be also a model to take account for a possible upper bound on the mass because it is concerned with a bound on the admissible velocity, which is decreasing with the mass concentration on the edge. In particular, looking to the function $U$ of the constraint (\ref{eq:constraint}), if $U(\rho)=0$ for $\rho\ge\rho_{\max}$ (the maximal mass), then the only admissible velocity is $u=0$ and so no agents can move: the edge is fully congested. Actually, here we have assumed that $U(\rho)>0$ for all $\rho$, and this fact was useful, for example, to prove that $t\mapsto\tau(t,e,\rho^e)$ is strictly increasing. However, we somehow get that fully congested property when $U(\rho_{\max})$ is sufficiently small in such a way that, if $\rho^e\sim\rho_{\max}$ in the time interval $[t,T]$, then $\tau(t,e,\rho^e)>T$ and so the agents do not move.
\end{remark}

\section{Conclusions}\label{sec:5}

In this paper we have modelled the agents flows over a transportation network via a  mean field game model which also takes into account the agents' preferences about the paths choice. {We proved the existence of  a mean field equilibrium, and also addressed the case where the set of admissible controls depends on the actual congestion of the edge}. 

Future research may be to study the behaviour of the mean field equilibrium by varying the noise to which the information on the congestion is subject and also to compare our mean field model with the Wardrop one, {with also some possible numerical simulations. Also the effects of the only estimated flows assumption (\ref{flusso}) on the discrepancy from a real model is worth analyzing.}

\medskip
\medskip

\end{document}